\documentclass[reqno,12pt]{amsart}
\usepackage{amsfonts,amssymb,amsthm,amsmath}
\usepackage{graphicx}
\usepackage{hyperref}
\makeatletter

\newtheorem{theorem}{Theorem}
\newtheorem{corollary}{Corollary}
\newtheorem{proposition}{Proposition}

\theoremstyle{definition}

\newtheorem{openpb}{Open Problem}

\numberwithin{equation}{section}

\makeatother

\def\R{\mathbb{R}}

\newcommand{\ovl}{\overline}
\newcommand{\prt}{\partial}

\newcommand\la{\lambda}

\newcommand\Om{\Omega}

\title[Boundedness of stable solutions to semilinear elliptic equations]
{Boundedness of stable solutions to semilinear elliptic equations: a survey}

\author{Xavier Cabr\'e}
\address{X. Cabr\'e\textsuperscript{1,2}:
\newline
\textsuperscript{1} Universitat Polit\`ecnica de Catalunya, Departament de Matem\`{a}tiques, Diagonal 647, 
08028 Barcelona, Spain
\newline
\textsuperscript{2} ICREA, Pg. Lluis Companys 23, 08010 Barcelona, Spain}
\email{xavier.cabre@upc.edu}

\thanks{The author is supported by MINECO grant MTM2014-52402-C3-1-P
and is part of the Catalan research group 2014 SGR 1083}

\subjclass[2010]{Primary 35K57, 35B65}
     
\begin{document}

\begin{abstract}
This article is a survey on boundedness results for stable solutions to semilinear elliptic problems.
For these solutions, we present the currently known $L^{\infty}$ estimates that hold for all nonlinearities. 
Such estimates are known to hold up to dimension 4. 
They are expected to be true also in dimensions 5 to 9, but this is still an open problem which has only been
established in the radial case.
\end{abstract}

\date{}
\maketitle

\tableofcontents

\section{Introduction. The known $L^\infty$ estimates}

This article is a survey on regularity results for \emph{stable} solutions of semilinear elliptic problems.
For these solutions, we present the currently known $L^{\infty}$ estimates that hold for all nonlinearities,
or at least for a large class of them. It is Ha\"{\i}m Brezis who stressed, since the mid-nineties, the significance 
of this problem.

Even though we will briefly comment at the end of this introduction on related results for other nonlinear 
elliptic problems, the rest of the paper concerns the semilinear problem
\begin{equation}\label{pb}
\left\{
\begin{array}{rcll}
-\Delta u &=& f(u) & \quad\mbox{in $\Omega$}\\
u &=& 0  & \quad\mbox{on $\prt\Omega$,}\\
\end{array}\right.
\end{equation}
where $\Omega \subset \R^n$ is a smooth bounded domain, $n\geq 1$, and $f$ is a locally Lipschitz nonlinearity.

Following \cite{BCMR}, 
we say that $u$ is a \emph{weak solution} of \eqref{pb} if, with $d_\Om=\textrm{dist}(\cdot,\partial\Om)$, we have
$u\in L^1(\Om)$, $f(u)d_\Om\in L^1(\Om)$, and $-\int_\Om
u\Delta\zeta\,dx=\int_\Om f(u)\zeta\,dx$  for all $\zeta\in C^1(\ovl \Om)$ with $\zeta=0$ on $\prt\Om$.

A weak solution $u$ of \eqref{pb} is said to be \emph{stable} if $f'(u)\in L^1_{\rm loc}(\Om)$ and
\begin{equation} \label{stability}
\int_{\Omega} f'(u)\xi^2\,dx
\leq
\int_{\Omega} |\nabla\xi|^2\,dx \quad\textrm{for all } \xi \in C^1_c(\Omega),
\end{equation}
i.e., for all $C^1$ functions $\xi$ with compact support in $\Om$.
If the solution $u$ is bounded, its stability is equivalent to the nonnegativeness of the first 
Dirichlet eigenvalue in~$\Om$ of the linearized operator $-\Delta-f'(u)$ for \eqref{pb} at $u$.
As a consequence, \emph{local minimizers} of the energy (i.e., minimizers under small perturbations) are stable
solutions ---whenever the energy and all objects are well defined within a proper functional analytic framework.

In dimensions 10 and higher, there are explicit examples of unbounded (also called singular) stable solutions. 
The following is one of them.

\begin{proposition}\label{proplog}
For $n\geq 3$, the function $\tilde u=-2\log\vert{x}\vert$ is a weak solution of \eqref{pb} in $\Om =B_1$, 
the unit ball, for the nonlinearity $f(u)=2(n-2)e^u$. For $n\geq 10$, $\tilde u$ is an $H^1_0(B_1)$
stable weak solution.
\end{proposition}

This result is easily proved. One verifies that  $\tilde u$ is a weak solution whenever $n\geq 3$. Next,  
the linearized operator at $\tilde u$ is given by  
$$
-\Delta -2(n-2)e^{\tilde u} =
-\Delta  -\frac{2(n-2)}{\vert{x}\vert^2}.
$$
Thus, if $n\ge 10$ then its first Dirichlet eigenvalue in $B_1$ is nonnegative.
This is a consequence of {\it Hardy's inequality}:
$$
\frac{(n-2)^2}{4}\int_{B_1}\frac{\xi^2}{\vert{x}\vert^2} \;\; \leq\;\;
\int_{B_1} \vert\nabla \xi\vert^2
\qquad\quad
\mbox{ for every } \xi\in H^1_0(B_1),
$$
and the fact that $2(n-2)\leq (n-2)^2/4$ if $n\geq 10$.

Thus, in dimensions $n\geq 10$ there exist unbounded $H^1_0$ stable solutions of \eqref{pb}, 
even in the unit ball and for the exponential nonlinearity. One may wonder if the same could
happen in a lower dimension, perhaps for another nonlinearity and in another domain.
The regularity results that we present next indicate that the dimension~$10$ should be optimal. 
However, \emph{a main open question is still to be answered
in dimensions $5\le n\le 9$} ---see Open Problem \ref{OP1} below. 
Let us emphasize that here we are talking about regularity results that hold for 
all nonlinearities $f$ ---or at least for a very large class of them. 
In particular, we will allow nonlinearities with supercritical growth with respect to the Sobolev exponent, 
as in Proposition~\ref{proplog}. It is therefore the stability of the solution that will lead to $L^\infty$ estimates. 

All this is reminiscent of important results within the theories of minimal surfaces and of harmonic maps (see
the survey \cite{CCsur} for these analogies), in which regularity of stable solutions also holds only in 
certain low dimensions.

There are many nonlinearities for which \eqref{pb} admits a stable weak solution $u$ with $u\not\equiv 0$. 
Indeed, replace $f(u)$ by $\lambda f(u)$
in \eqref{pb}, with $\lambda\geq 0$. That is, consider the problem
\begin{equation}\label{pbla}
\left\{
\begin{array}{rcll}
-\Delta u &=& \la f(u) & \quad\mbox{in $\Omega$}\\
u &=& 0  & \quad\mbox{on $\prt\Omega$.}\\
\end{array}\right.
\end{equation}
Assume that $f$ is positive, nondecreasing, and superlinear at $+\infty$, that is,
\begin{equation} \label{hypf}
f(0) > 0, \quad f'\geq 0 \quad\textrm{and}\quad\lim_{t\to +\infty}\frac{f(t)}t=+\infty.
\end{equation}
Note that in this case we look for positive solutions (when $\la>0$).

Assuming \eqref{hypf}, there exists an
extremal parameter $\lambda^*\in (0,+\infty)$ such that if $0\le\lambda<\lambda^*$ then
\eqref{pbla} admits a minimal classical solution $u_\lambda$. Here minimal means smallest. 
On the other hand,
if $\lambda>\lambda^*$ then \eqref{pbla} has no classical solution.
The family of classical solutions  $\left\{ u_{\lambda }:0\le \lambda <\lambda ^{*}\right\}$ 
is increasing in $\lambda$, and its limit as 
$\lambda\uparrow\lambda^{*}$ is a weak solution $u^*=u_{\lambda^*}$ 
of \eqref{pbla} for $\la=\la^*$. $u^*$ is called \emph{the extremal solution} of \eqref{pbla}.
Moreover, for every $\la\in [0,\la^*]$, $u_\la$ is a \emph{stable} weak solution of \eqref{pbla}. 

The proofs of the previous statements can be found in the book \cite{Dupaigne}
by L. Dupaigne; see chapter~3 for these issues. There, among many others, two further interesting results are proved
under the hypothesis that $f$ is convex and satisfies \eqref{hypf}. First,
the nonexistence of \emph{weak} solutions of \eqref{pbla} for 
$\la>\la^*$ ---a result established in a seminal paper in this topic, \cite{BCMR}, by Brezis, Cazenave, Martel,
and Ramiandrisoa. Second, Martel's result stating that $u^*$ is the unique weak solution of \eqref{pbla} for 
$\la=\la^*$.  The book by Dupaigne also contains proofs
of the statements and estimates that we present below. Another nice monograph in this
subject is due to J. D\'avila~\cite{Davila}.

\subsection{$L^\infty$ estimates}
We now turn to the core of the paper: $L^\infty$ estimates for stable solutions. As usual, the key
point will be to establish \emph{apriori} $L^\infty$ estimates for stable \emph{classical} solutions
of \eqref{pb}. Once this is done, they lead to the boundedness (and thus regularity) of the extremal
solution $u^*$ of \eqref{pbla} ---simply by approximation of $u^*$ by the stable classical solutions
$u_\la$ with $\la\uparrow\la^*$, since the apriori estimates will be uniform in $\la$ as $\la\uparrow\la^*$. 
In this respect, we could now forget about the extremal problem \eqref{pbla}, come back to equation \eqref{pb},
and ask whether a stable weak solution of \eqref{pb} can be approximated by classical solutions
to similar problems. A nice answer was given in \cite{BCMR} (see also Corollary 3.2.1 of 
the monograph~\cite{Dupaigne})
and states that this is always possible if $u\in H^1_0(\Om)$ is a stable weak solution of \eqref{pb}
and $f$ is nonnegative and convex. The approximating functions can be taken to be stable classical solutions
to problem \eqref{pb} with $f$ replaced by $(1-\varepsilon)f$, where $\varepsilon\downarrow 0$.

There is a large literature on a priori estimates for stable solutions, beginning with the seminal
paper of Crandall and Rabinowitz \cite{CR} from 1975. In more recent years there have been strong 
efforts to obtain a priori bounds under minimal assumptions on $f$, mainly after
Brezis and V\'azquez \cite{BV} raised several open questions ---see also the open
problems raised by Brezis in \cite{Brezis}. 

The following three theorems collect all what is currently
known in terms of $L^\infty$ estimates. We will give their proofs
(with most of details) in the next three sections.

We start with a result that suggests the sharpness of dimension 10.
However, it requires the nonlinearity $f$ to be {\it convex}, satisfy \eqref{hypf}, and also that
\begin{equation}\label{tau}
f\in C^{2}(\R) \text{ and the limit } \lim_{t \to +\infty} \frac{f(t)f''(t)}{f'(t)^2} \text{ exists.}
\end{equation}
It is easy to see that if the limit exists, then it must belong to $[0,1]$ ---otherwise~$f$ would blow up somewhere 
in $(0,+\infty)$. These hypotheses are satisfied by many nonlinearities, including 
$f(u)=e^u$, $f(u)=e^{e^{u}}$, as well as $f(u)=(1+u)^p$ with $p>1$.

\begin{theorem}[{\cite{CR,MP,Sanchon,JL}}]
Let $u^*$ be the extremal solution of \eqref{pbla}.

$(a)$ Let $f$ be convex and satisfy \eqref{hypf} and \eqref{tau}. If $n\le 9$, 
then $u^*\in L^\infty(\Om)$.

$(b)$ If $\Om=B_1$, $f(u)=e^u$, and $n\geq 10$, then $u^*=-2\log \vert x\vert$ and $\la^*=2(n-2)$.
\label{thm1}
\end{theorem}

Note that in part (a) we may take $f(u)=e^u$. Together with part (b) ---that contains more information than 
Proposition~\ref{proplog}--- gives a quite detailed answer to the boundedness of the extremal solution, 
in any dimension, for the exponential nonlinearity.  Part (b) was
first established in \cite{JL} using ODEs techniques, while a more flexible PDE proof was given 
in \cite{BV}.

Part (a) also applies to $f(u)=(1+u)^p$ with $p>1$. For these power nonlinearities, the boundedness of stable solutions
holds indeed for $n\leq 10$ ---and also for $n\geq 11$ and $p$ smaller than the Joseph-Lundgren exponent.
For the rest of couples $(n,p)$, the explicit radial solution is given by 
$|x|^{-2/(p-1)}-1$ and coincides with the extremal solution $u^*$ when $\Om=B_1$. 
See \cite{BV,CC} for more details on this power case.

Part (a) of Theorem \ref{thm1} was proven by
Crandall and Rabinowitz \cite{CR} and by Mignot and Puel \cite{MP}, not only for the exponential nonlinearity
but also for some others. The improved result under assumption \eqref{tau} is due to Sanch\'on~\cite{Sanchon}. 
The proofs use the equation in \eqref{pbla} together with the stability condition \eqref{stability} applied to
\begin{equation}\label{choice1}
\xi:= (f(u)-f(0))^\alpha \quad\text{ or }\quad \xi:= f(u)^\alpha- f(0)^\alpha
\end{equation}
for a well chosen positive exponent $\alpha$. In section~ 2, we will present the proof with 
all details in the case $f(u)=e^u$.

Note that assumption \eqref{tau} on a convex nonlinearity $f$ prevents $f''\ge 0$ to oscillate at infinity.
Think that $f$ could be linear within some intervals approaching infinity, and grow exponentially 
in the complement of the intervals.
We will see that to establish regularity for stable solutions without the additional assumption 
\eqref{tau} on~$f$ is a much harder task. 

However, the radial case $\Om=B_1$ was settled by the author and Capella~\cite{CC} in 2006,
as we state next. Note that the estimates apply to \emph{every nonlinearity} $f$. This work 
does not use phase plane analysis, but PDE techniques instead.

\begin{theorem}[{\cite{CC}}]\label{thm2}
Let $n\ge 1$ and $f:\R\to\R$ be a locally Lipschitz function. 

Let $u\in H^1_0(B_1)$ be a stable radial weak solution of \eqref{pb}.
Then, $u$ is either constant, radially increasing, or radially decreasing in $B_1$. Moreover,

$(a)$ If $n\leq 9$, then $u\in L^\infty (B_1)$ and
\begin{equation}\label{linftyrad}
\left\Vert u\right\Vert _{L^{\infty}(B_{1})}\leq C_{n}
\left( \, \Vert u\Vert _{L^{1}(B_{1})}
+\Vert f(u)d_{B_1}\Vert_{L^{1}(B_{1})}\, \right)
\end{equation}
for some constant $C_n$ depending only on $n$, where $d_{B_1}={\rm dist}(\cdot,\partial B_1)$.

$(b)$ If $n=10$, then $|u(|x|)|\leq C |\log |x|\,|$ in $B_1$
for some constant $C$.

$(c)$ If $n\geq 11$, then $|u(|x|)|\leq C  |x|^{-n/2+\sqrt{n-1}+2} 
|\log |x|\,|^{1/2}$ in $B_1$ for some constant~$C$.
\end{theorem}

In this result, the assumption that the weak solution belongs to $H^1_0(B_1)$ is needed
---see Example 3.2.1 of~\cite{Dupaigne} for this statement. 
The proof of the theorem will be given in section~3, but later in this introduction we will comment on its proof.

From the theorem, and working with the 
stable classical solutions $u=u_\la$ of \eqref{pbla}, it is easy to bound the right-hand side of \eqref{linftyrad}
uniformly in $\lambda$ ---assuming $\la>\la^*/2$ and 
\eqref{hypf} on the nonlinearity. Indeed, one first multiplies \eqref{pbla}
by the first Dirichlet eigenfunction $\varphi_1$ of the Laplacian. We obtain 
$\lambda_1\int_\Om u\varphi_1\,dx =\lambda\int_\Om f(u)\varphi_1\,dx$. Using now the superlinearity of $f$ at
infinity, one deduces a bound for $\int_\Om f(u)\varphi_1\,dx$ uniform in $\la>\la^*/2$. 
From this, multiplying \eqref{pbla}
by the solution $w$ of $-\Delta w= 1$ in $\Om$ with zero Dirichlet boundary values, and controlling $w$ by
$\varphi_1$, one finally bounds $\Vert u \Vert_{L^1(\Om)}$ uniformly as $\la\uparrow\la^*$. 
Therefore, since the classical solutions $u=u_\la$ are radially decreasing by the moving planes method, 
from the previous theorem we deduce:

\begin{corollary}[\cite{CC}]
\label{corolrad}
Assume that $f$ satisfies \eqref{hypf} and that $\Om=B_1$. Let $u^*$ be the extremal solution of~\eqref{pbla}.
If $1\leq n\leq 9$, then $u^{*}\in L^\infty (B_1)$. 
\end{corollary}

We turn now to the nonradial case. We present the currently known estimates that hold for all nonlinearities.
These estimates agree with the available ones for convex nonlinearities satisfying \eqref{hypf} ---nothing
better is known in this case. An apriori $L^\infty$ estimate is known only up to dimension 4.

\begin{theorem}\label{thm3}
$(a)$ {\rm (\cite{Cabre})}
Let $f$ be any smooth function satisfying $f\geq 0$ in $\R$ and  
$\Omega\subset \mathbb{R}^n$ a smooth bounded domain.
Assume that $2\leq n\leq 4$ and let $u$ be a stable classical solution of~\eqref{pb}. 
Then, for any compact set $K\subset \Om$, we have
\begin{equation}\label{linf4pb}
 \Vert u\Vert_{L^\infty(\Omega)} \leq C (\Omega, K, \Vert u\Vert_{L^1(\Om)}, 
 \Vert \nabla u\Vert_{L^4(\Omega\setminus K)}),
\end{equation}
where $C$ is a constant depending only on the quantities within the parentheses.

$(b)$ {\rm (\cite{Nedev,Cabre,V})}
Let $f$ satisfy \eqref{hypf}, $\Omega\subset\R^{n}$ be a smooth bounded domain, and $1\leq n\leq 4$.
If $n\in \{3,4\}$ assume either that $f$ is a convex nonlinearity or that $\Omega$ is a convex domain. 
If $u^*$ is the extremal solution of \eqref{pbla}, then $u^* \in L^\infty(\Omega)$.
\end{theorem}

Note that part (b) of the theorem does not require assumption \eqref{tau} on $f$.

In dimensions 2 and 3, and for convex nonlinearities, part (b) was proved
by G.~Nedev~\cite{Nedev} in 2000. We will briefly describe its proof in section~2.

Ten years later the author established the first result in dimension 4, in \cite{Cabre}.
Note that part (a) holds for {\it all nonnegative 
nonlinearities}, and that the $L^{4}$ norm of $\nabla u$ is computed in an arbitrarily small neighborhood of
the boundary $\partial\Om$. As a consequence, we will see that the estimate leads easily
to the boundedness of the extremal solution in convex domains. This is part of the 
statement in (b). 

In part (b), for $3\leq n \leq 4$, \cite{Cabre} requires $\Omega$ to be convex, while $f$ does not need to be
convex. Some years later, S. Villegas~\cite{V} succeeded to use both \cite{Cabre} and \cite{Nedev} when $n=4$
to remove the requirement that $\Omega$ is convex by further assuming that $f$ is convex. 

\begin{openpb}\label{OP1}
For convex nonlinearities satisfying \eqref{hypf}, the boundedness of the extremal solution $u^*$,
or the existence of an apriori $L^\infty$ bound for stable classical solutions, is still unknown
in dimensions $5\leq n\leq 9$. This is the case even for convex domains which are symmetric with respect
to each coordinate hyperplane $\{x_i=0\}$. It is only known, by Theorem~\ref{thm2}, in balls.
\end{openpb}

Let us explain where the convexity of $\Omega$ comes into play in Theorem~\ref{thm3}
and raise an open problem in this direction. As we showed after Theorem~\ref{thm2},  
for a solution $u$ of \eqref{pb} and for nonlinearities $f$ which are superlinear at infinity, 
one can give a bound for $\Vert u \Vert_{L^1(\Om)}$ depending only on $f$ and $\Omega$. 
Therefore, estimate \eqref{linf4pb}
in Theorem~\ref{thm3} reduces the $L^\infty$ bound for stable solutions $u$ to a bound for $u$
in a sufficiently small neighborhood $\ovl\Om\setminus K$ of the boundary. In this case we need an $L^4$ estimate
for $\nabla u$ in $\Om\setminus K$. Thus, the question is now whether there 
are boundary estimates for stable positive solutions
of semilinear elliptic problems. This turns out to be true in convex domains (and here the stability of the 
solution is not needed), since the moving planes method
gives a cone of monotonicity for $u$ at all points close enough to $\partial\Om$. Such cone of monotonicity
gives automatically an $L^\infty$ control for $u$ near $\partial\Om$ in terms of 
$\Vert u \Vert_{L^1(\Om\setminus K)}$ ---a quantity that, as $\Vert u \Vert_{L^1(\Om)}$, is under control. 
See \cite{Cabre,Dupaigne} for more details.

In dimension 2, W. X. Chen and C. Li \cite{CL} extended this boundary estimate to general domains 
---they removed the convexity assumption
on the domain after doing a Kelvin transform at boundary points. One only needs to assume that $f\geq 0$;
see \cite{CL,Cabre} for more details. However, in dimensions $n\geq 3$ boundary estimates 
for stable solutions remain open:

\begin{openpb}\label{OP2}
Are there apriori {\it boundary estimates} (for instance $L^\infty$ bounds in a neighborhood of the boundary) 
for stable classical solutions in nonconvex domains $\Om$
of $\R^n$, $n\geq 3$, for all nonnegative nonlinearities $f$ or at least for a large class of them? 
It would be desirable that the control is in terms of $\Vert u \Vert_{L^1(\Om)}$.

In this respect, Nedev's estimate \eqref{boundnedev}, presented below, is a global integral bound ---and thus also
a boundary integral bound. In section~4 it will be used to control, when $n\leq 4$, the ``boundary quantity'' 
$\Vert\nabla u\Vert_{L^4(\Om\setminus K)}$ in \eqref{linf4pb}.
\end{openpb}

Let us finally comment on proofs. Recall \eqref{choice1}:
for Theorem~\ref{thm1} the key point is to use a power of $f$ as test function $\xi$ in the stability condition
\eqref{stability}. 
Instead, for Theorems \ref{thm2} and \ref{thm3} we will use other type of test functions. 
The proofs start by writing 
the stability condition \eqref{stability} for the test 
function $\xi=c\eta$, where $\eta|_{\partial\Omega}\equiv0$. 
Integrating by parts, one easily deduces that
\begin{equation}\label{stabc}
\int _{\Omega} \left( \Delta c+f'(u)c\right) c\eta ^{2}\,dx \le
\int _{\Omega} c^{2}\left|\nabla \eta\right|^{2}\,dx.
\end{equation}
Next, a key point is to choose a function $c$ satisfying an appropriate equation for
the linearized operator $\Delta + f'(u)$. In the radial case the choice of $c$ and the final choice of $\xi$ are
\begin{equation}\label{choicerad}
c=u_{r} \quad\text{ and }\quad\xi=u_r r(r^{-\alpha}-(1/2)^{-\alpha})_{+}
\end{equation}
where $r=|x|$, $\alpha>0$, and $\xi$ is later truncated arbitrarily near the origin to make it Lipschitz.
Note that, differentiating $u_{rr}+(n-1)r^{-1}u_r+f(u)=0$ with respect to~$r$, we get
\begin{equation}\label{linrad}
(\Delta + f'(u)) \, u_r=(n-1)\frac{u_r}{r^2}.
\end{equation}

This way of proceeding was motivated by the study of stable minimal cones in the theory of 
minimal surfaces, where 
one takes $c=\left|A\right|$ (the norm of the second fundamental 
form of the cone), and then later $\xi= c r^{-\alpha}$ (properly truncated), where $r$~is 
the distance to the vertex of the cone.

For the estimate up to dimension 4 in the nonradial case, Theorem~\ref{thm3}, we will take
\begin{equation}\label{choice4}
c=\left|\nabla u\right| \quad\text{ and }\quad\xi=\left|\nabla u\right|\varphi(u)
\end{equation}
where, in dimension $n=4$, $\varphi$ will be chosen depending on the solution $u$ itself 
---thus not depending only on $f$
as in Theorem~\ref{thm1}.
A simple computation (see section~4) shows that
\begin{equation}\label{eq:grad}
\left( \Delta+f'(u)\right)\left|\nabla u\right|=
\frac{1}{\left|\nabla u\right|} \left(\left|\nabla_T\left|\nabla u\right|\right|^2+
\left|A\right|^2\left|\nabla u\right|^2\right)\quad \text{in}\ 
\Omega\cap\left\{\left|\nabla u\right|>0\right\},
\end{equation}
where $\left|A\right|^2=\left|A\left(x\right)\right|^2$ is the squared norm of 
the second fundamental form of the level set of $u$ passing through a 
given $x\in\Omega\cap\left\{\left|\nabla u\right|>0\right\}$, i.e., the 
sum of the squares of the principal curvatures of the level set. On the 
other hand, $\nabla_T$ denotes the tangential gradient to the level set. 
Thus, \eqref{eq:grad} involves geometrical information of the level sets 
of $u$. As we will see in section~4, the choice of the function $\varphi=\varphi(t)$ 
in \eqref{choice4} will depend on certain integrals on the level set $\{u=t\}$. 

The next crucial point in dimensions $n\leq 4$ is the use of the {\it Michael-Simon 
and Allard Sobolev inequality} (Proposition~\ref{Sobolev} below) 
in every regular level set of $u$. This is a remarkable 
Sobolev inequality on general hypersurfaces of $\mathbb{R}^n$ 
that involves their mean curvature but has a best constant that is independent of the 
hypersurface ---it depends only on the dimension~$n$.

\subsection{Other nonlinear elliptic problems}
The previous questions have been studied also for semilinear elliptic problems involving
other operators. For the $p$-Laplace operator and $f(u)=e^u$, minimal and extremal solutions of \eqref{pbla} 
were studied by Garc\'{\i}a-Azorero,
Peral, and Puel \cite{GP,GPP}. They established the
boundedness of the extremal solution when $n<p+4p/(p-1)$, and showed
that this condition is optimal. Under hypotheses \eqref{hypf} and \eqref{tau}, 
Theorem~\ref{thm1} has been extended
to the $p$-Laplace case by Sanch\'on~\cite{Sanchon}. The radial case in Theorem~\ref{thm2}
has been extended by Capella, Sanch\'on, and the author~\cite{CCaS}.

Other works treat the case of the bilaplacian, the mean curvature operator for graphs, 
the fractional Laplacian, 
boundary reaction problems (this is related with fractional diffusion), and elliptic systems.
Optimal $L^\infty$ bounds are still unknown for many of these problems, in particular for the fractional
Laplacian. 

Finally, the semilinear equation $-\Delta u=\la f(u)$ has also been studied for nonlinearities $f:[0,1)\to\R$
which blow-up at $u=1$ ---this is in connection with micro-electro-mechanical devices, MEMS; 
see the monograph \cite{EGG}. 

For more information and references in all these problems, see 
Dupaigne's monograph~\cite{Dupaigne} ---mainly chapter~8.

\section{The exponential nonlinearity. Nedev's result}

The proof of Theorem~\ref{thm1} uses the stability condition with the test function $\xi$
being essentially a power of $f(u)$. Let us give all details 
of the proof in the case of the exponential nonlinearity.

\begin{proof}[Proof of Theorem~\ref{thm1} (a) when $f(u)=e^u$]
Let $u$ be a stable solution of \eqref{pbla} with  $f(u)=e^u$. Take $\la\in (0,\la^*)$ so that $u$ is a 
classical solution, and $\xi = e^{\alpha u} -1$ in the stability
condition \eqref{stability}. We deduce
\begin{eqnarray*}
\int_{\Om} \la e^u (e^{\alpha u} -1)^2\,dx & \le & \int_{\Om} \alpha^2 e^{2\alpha u} |\nabla u|^2\,dx
=\frac{\alpha}{2}  \int_{\Om} \nabla (e^{2\alpha u}-1) \cdot \nabla u\,dx\\
& = & \frac{\alpha}{2}  \int_{\Om} (e^{2\alpha u}-1) (-\Delta u)\,dx
= \frac{\alpha}{2}  \int_{\Om} \la e^u (e^{2\alpha u}-1)\,dx,
\end{eqnarray*}
where we have used that $u$ solves problem \eqref{pbla}. {From} the first and last integrals
in these inequalities, we deduce that
$$
\left( 1-\frac{\alpha}{2}\right) \int_{\Om} e^{(1+2\alpha)u}\,dx
\le 2 \int_{\Om} e^{(1+\alpha)u}\,dx.
$$

Taking any $\alpha < 2$ and using H\"older inequality, this allows to bound $ \int_{\Om} e^{(1+2\alpha)u}\,dx$
by a constant depending only on $\alpha$ and $|\Om|$. Denoting $p:=1+2\alpha$, we conclude
that $\la e^u \in L^p(\Om)$ for every $p<5$ ---with an apriori bound independent of $\la$.
Thus $-\Delta u$ is in $L^p (\Om)$ and, hence,
$u$ is controlled in $W^{2,p}(\Om)$ uniformly in $\la$.  
If $n\le 9$ the exponent~$p$ can be chosen such that $n/2 <p<5$
and, therefore, we deduce a uniform $L^\infty(\Om)$ bound for $u$ by the Sobolev embedding.
\end{proof}

Nedev extended the previous argument to give an $L^\infty$ bound in dimensions 2 and 3
for all convex nonlinearities satisfying \eqref{hypf}. 
His proof uses the stability condition with
$\xi= f(u)-f(0)$, and also tests the equation in \eqref{pbla}
multiplying it by $g(u) := \int_0^u f'(t)^2\,dt$. After some clever bounds on 
several nonlinearities which come up in the argument (see \cite{Nedev,Dupaigne}), 
he deduces the apriori bound
\begin{equation}\label{boundnedev}
\int_\Om f'(u)f(u) \,dx\le C. 
\end{equation}
Next, since $f$ is convex, we have $-\Delta (f(u)-f(0))\le f'(u)(-\Delta u)=\la f'(u)f(u)$.
Thus $0\le f(u)-f(0)\le v$, where $v$ is the solution of $-\Delta v= \la f'(u)f(u)$ with
zero Dirichlet boundary values. From \eqref{boundnedev} we deduce that $v$, and hence 
also $f(u)-f(0)$, belong to $L^p(\Om)$ for all $p<n/(n-2)$. We conclude that
$u$ is bounded in $W^{2,p}(\Om)$ for all $p<n/(n-2)$. This gives an $L^\infty$ estimate
in dimensions 2 and 3.

An alternative proof of Nedev's bound \eqref{boundnedev} was recently found 
by Sanch\'on, Spruck, and the author \cite{CSS}.

\section{The radial case}

We give here the main ideas in the proof of Theorem~\ref{thm2} concerning the radial case.
See \cite{CC,Dupaigne} for the proof with all details.
As stated in the theorem, the proof requires $u\in H^1_0(B_1)$ ---a condition which
is needed by Example 3.2.1 of~\cite{Dupaigne}. Instead, the equation may be assumed to hold only in 
$B_1\setminus\{0\}$, since the test functions in the following argument can be cut-off
near the origin.

We use the stability in the form
\eqref{stabc} with $c=u_{r}$, together with the linearized equation~\eqref{linrad}, to obtain
\begin{equation}\label{radst}
(n-1)\int_{B_1}u_r^2\, \frac{\eta^2}{r^2}\,dx
\leq\int_{B_1}u_r^2 |\nabla \eta|^2\,dx. 
\end{equation}
Choose now $\eta=r(r^{-\alpha}-(1/2)^{-\alpha})_{+}$ ---after
being cut-off near the origin to make it Lipschitz. In the following, however,
we do not give all the details of the cut-off argument, which are standard. 
We deduce
\begin{equation}\label{poincrad}
(n-1)\int_{B_1}u_r^2(r^{-\alpha}-(1/2)^{-\alpha})_{+}^2\, dx
\leq\int_{B_1}u_r^2 |\nabla \left( r(r^{-\alpha}-(1/2)^{-\alpha})_{+}\right)|^2 \,dx.
\end{equation}
The most singular term near the origin in the right-hand side is 
$(\alpha-1)^2 \int_\Om u_r^2r^{-2\alpha}\,dx$. The same integral appears as most singular term in the
left-hand side, but with constant $n-1$ in front. After a simple interpolation argument, we deduce that
\begin{equation}\label{ineqrad}
\int_{B_{1/2}}u_r^2r^{-2\alpha}\,dx\leq C_{\alpha,n} \int_{B_{1/2}}u_r^2r^{n-1}\,dx
\end{equation}
if $(\alpha -1)^2 < n-1$. That is, we need to take $\alpha$ such that
\begin{equation}
1\le \alpha<1+\sqrt{n-1}.
\label{eq:alpha}
\end{equation}

Using that $u$ is radial, it is not difficult to see that the integral 
in the right-hand side of \eqref{ineqrad} can be bounded by the right-hand side
of \eqref{linftyrad}. With this at hand, the bound in \eqref{ineqrad},
valid for the above range of exponents $\alpha$, leads to an $L^\infty$ estimate for $u$ whenever
$n\leq 9$. To see this we just notice that, for $0<s\le 1/2$, we have
\begin{eqnarray}\label{eq:23}
 |u(s)-u(1/2)| & = & \left|\int _{s}^{1/2}-u_{r}\,dr\right|
\le \int _{s}^{1/2} |u_{r}|r^{-\alpha +\frac{n-1}{2}}
r^{\alpha -\frac{n-1}{2}}\,dr\nonumber\\
 & \leq  & C_n \left( \int _{B_{1/2}}u_{r}^{2}
r^{-2\alpha }\,dx\right)^{1/2}
\left(\int _{s}^{1/2}r^{2\alpha +1-n}\,dr\right)^{1/2}
\end{eqnarray}
by Cauchy-Schwarz. 

The last integral in \eqref{eq:23} is finite with $s=0$ if we take $2\alpha +1-n>-1$, i.e.,
\begin{equation*}
(n-4)/2<\alpha-1 .
\end{equation*}
Since $n\leq 9$, then $(n-4)/2<\sqrt{n-1}$ and we can choose
$\alpha $ satisfying this last condition as well as \eqref{eq:alpha}.

The pointwise estimates in dimensions $n\geq 10$ stated in Theorem~\ref{thm2} also follow from  
\eqref{ineqrad}, used together with \eqref{eq:23}. These 
pointwise bounds give $L^q$ integrability results for $H^1_0$ stable radial solutions with a range of
exponents $q$ depending only on the dimension $n\geq 10$. The range turns out to be optimal,
as it is checked with the explicit radial solutions $r^{-2/(p-1)}-1$ (see \cite{CC}).

\section{Dimensions $n\leq 4$}

In this section we give most of the details of the proof of Theorem~\ref{thm3};
see \cite{Cabre,Dupaigne} for all details. The key estimate is the following:

\begin{theorem}[\cite{Cabre}]\label{thm4}
Let $f$ be any smooth function and
$\Omega\subset \mathbb{R}^n$ any smooth bounded domain.
Assume that $2\leq n\leq 4$.  Let $u$ be a positive stable classical solution of \eqref{pb}.

Then, for every $t>0$,
\begin{equation}\label{est4}
 \Vert u\Vert_{L^\infty(\Omega)} \leq t + \frac{C}{t}|\Omega|^{(4-n)/(2n)}
\left( \int_{\{0<u <t\}} |\nabla u|^4 \ dx\right)^{1/2},
\end{equation}
where $C$ is a universal constant {\rm (}in particular, independent of 
$f$, $\Omega$,~and~$u${\rm)}.
\end{theorem}

To establish this result it will be essential to apply on every regular level set of $u$ the 
following remarkable Sobolev inequality.

\begin{proposition}[Allard; Michael and Simon; see \cite{Cabre,Dupaigne}]
\label{Sobolev}
Let  $M\subset  \R^{m+1}$ be an immersed smooth $m$-dimensional 
compact hypersurface without boundary.

Then, for every $p\in  [1, m)$,
there exists a constant $C = C(m,p)$ depending only on the dimension $m$ 
and the exponent $p$ such that, for every $C^\infty$ function $v : M  \to \R$, 
\begin{equation}\label{MSsob}
\left( \int_M |v|^{p^*} \,dV \right)^{1/p^*} \leq C(m,p)
\left( \int_ M (|\nabla v|^p +  |Hv|^p) \,dV \right)^{1/p},
\end{equation}
where $H$ is the mean curvature of $M$ and $p^* = mp/(m - p)$.
\end{proposition}

Note that when $M$ is a minimal surface and $p=1$, \eqref{MSsob} is equivalent to
a universal isoperimetric inequality on subsets of any minimal surface. 

Let us start showing how to use Theorem~\ref{thm4} to deduce part (a) of Theorem~\ref{thm3}, as well as part (b) 
of Theorem~\ref{thm3} in convex domains. After this, we will explain the proof of part (b) of Theorem~\ref{thm3} 
for convex nonlinearities. Finally, we will establish Theorem~\ref{thm4}.

\subsection{Proof of Theorem~\ref{thm3}: part (a), and part (b) in convex domains}
We use a well known estimate for a positive solution of \eqref{pb}
(note that $u$ is positive since we assume $f\geq 0$).
We have
$$
u \ge c_\Om\left( \int_\Om f(u)d_\Om \,dx\right) d_\Om \qquad\text{in } \Om,
$$
where $d_\Om={\rm dist}(\cdot,\partial\Om)$ and $c_\Om$ is a constant depending only on $\Om$;
see Proposition~A.4.2 in \cite{Dupaigne}. We deduce that, in \eqref{est4},
$\{0<u<t\}\subset\{0<d_\Om< s\}$, where $s=t/(c_\Om \Vert f(u)d_\Om\Vert_{L^1(\Om)})$. 
Now, as explained right after Theorem~\ref{thm2},
$\Vert f(u)d_\Om\Vert_{L^1(\Om)}$ can be bounded by below (and by above) by $\Vert u\Vert_{L^1(\Om)}$.

Therefore, given the compact set $K$ in part (a) of Theorem~\ref{thm3}, we simply need to
take $t$ small enough such that $\{0<d_\Om <s\}\subset \Om\setminus K$. Part (a) is now proved.

{From} part (a) we easily deduce part (b) when $\Omega$ is convex. First, as we mentioned in the introduction,
there are apriori bounds for positive solutions (not necessarily stable) in convex domains, 
thanks to the moving planes method; see Proposition~3.2 in~\cite{Cabre} or Theorem~4.5.3 in~\cite{Dupaigne}. 
For $n=2$ the convexity of $\Omega$ is not needed. We apply these boundary bounds
to the minimal solutions $u=u_\la$ of \eqref{pbla}
with $\la<\la^*$. We deduce a bound for 
$\Vert u\Vert_{L^\infty(\{d_\Om<\varepsilon\})}$ by a constant depending only on $\Omega$ and
$\Vert u\Vert_{L^1(\Om)}$, where $\varepsilon>0$ depends only on $\Omega$. As a consequence,
$\Vert f(u)\Vert_{L^1(\{d_\Om<\varepsilon\})}$ is also under control by a constant that depends
on $f$ in addition. All bounds are, in particular, uniform in $\la$.

Next, we use standard elliptic regularity theory to control
$\Vert \nabla u\Vert_{L^4(\{d_\Om<\varepsilon/2\})}$ by $\Vert u\Vert_{L^\infty(\{d_\Om<\varepsilon\})}$
and $\Vert f(u)\Vert_{L^\infty(\{d_\Om<\varepsilon\})}$. We now use part (a) of the theorem 
applied with the compact set $K:=\{d_\Om\ge\varepsilon/2\}$. Finally, as explained right after 
Theorem~\ref{thm2}, the superlinearity of $f$ gives a bound for $\Vert u\Vert_{L^1(\Om)}$
which is uniform in $\la>\la^*/2$.

\subsection{Proof of Theorem~\ref{thm3}: part (b) for convex nonlinearities}
We explain here the argument of S. Villegas~\cite{V} to remove the hypothesis that $\Om$
is convex when the dimension $n=4$, at the price of assuming that $f$ is convex and satisfies
\eqref{hypf}. See~\cite{V} for all details.

We consider the minimal solution $u=u_\la$ of \eqref{pbla} with $\la<\la^*$.
Since $f$ is convex, we have that $f(u)-f(0)\le uf'(u)$. Thus, Nedev's estimate~\eqref{boundnedev}
gives a control on $\int_\Om (f(u)-f(0))f(u)/u\, dx$. From this, and since $\lim_{t\to +\infty}f(t)=+\infty$,
it is easy to deduce a uniform bound 
\begin{equation}\label{villbound}
\int_{\{u>1\}} \frac{f(u)^2}{u}\,dx\le M.
\end{equation}

Taking $t:=\Vert \nabla u\Vert_{L^4(\Om)}$ in Theorem~\ref{thm4}, we deduce
$\Vert u\Vert_{L^\infty(\Om)}\leq C \Vert \nabla u\Vert_{L^4(\Om)}$. Therefore, using that $n=4$
in the Sobolev embedding, 
\begin{eqnarray*}
\Vert u\Vert_{L^\infty(\Om)} &\leq & C\Vert \nabla u\Vert_{L^4(\Om)} \le C\Vert u\Vert_{W^{2,2}(\Om)}
\le C\Vert f(u) \Vert_{L^2(\Om)} \\
&\le &C \left( f(1)^2|\Om|+\int_{\{u>1\}} f(u)^2\, dx\right)^{1/2}\le C \left( 1
+\int_{\{u>1\}} \frac{f(u)^2}{u} u \,dx\right)^{1/2}\\
&\le &  C \left( 1+M \Vert u\Vert_{L^\infty(\Om)} \right)^{1/2},
\end{eqnarray*}
where we have used \eqref{villbound}. From this, we conclude a bound for $\Vert u\Vert_{L^\infty(\Om)}$
that is uniform in $\la$.

\subsection{Proof of Theorem~\ref{thm4}}

We make the choice \eqref{choice4} and, in particular, we take
$c=\left|\nabla u\right|$ in \eqref{stabc}. It is easy to check that, in the set 
$\left\{\left|\nabla u\right|>0\right\}$, we have
\begin{equation*}
\left(\Delta + f'(u)\right)|\nabla u|=\frac{1}{|\nabla u|}
\left(\sum_{i,j}u_{ij}^2-\sum_i\left(\sum_{j}u_{ij}\frac{u_j}
{|\nabla u|}\right)^2\right).
\end{equation*}
Taking an orthonormal basis in which the last vector is the normal $\nabla u/|\nabla u|$ to the level set
and the other vectors are the principal directions of the level set, one easily sees that the last expression
coincides with the right-hand side of \eqref{eq:grad}. Therefore, using the stability condition \eqref{stabc},
we conclude the following result of Sternberg and Zumbrun. 

\begin{proposition}[Sternberg and Zumbrun; see \cite{Cabre,Dupaigne}]
\label{Prop-semi-st}
Let $\Omega$ be a smooth bounded domain of $\R^n$ and $u$ a smooth positive 
stable solution of \eqref{pb}. Then, for every Lipschitz function 
$\eta$ in $\overline\Omega$ with $\eta|_{\partial\Omega}\equiv0$,
\begin{equation}\label{semi1}
\int_{\left\{\left|\nabla u\right|>0\right\}} 
\left( |\nabla_T |\nabla u||^2 +|A|^2|\nabla u|^2\right)\eta^2\,dx
\leq \int_\Omega |\nabla u|^2 |\nabla \eta|^2 \,dx, 
\end{equation}
where $\nabla_T$ denotes the tangential or Riemannian gradient along a 
level set of $u$ and
$$
|A|^2=|A(x)|^2=\sum_{l=1}^{n-1} \kappa_l^2,
$$
with $\kappa_l$ being the principal curvatures of the level set of $u$ passing
through $x$, for a given $x\in\Omega\cap\left\{\left|\nabla u\right|>0\right\}$.
\end{proposition}

Note that in the radial case, \eqref{semi1} reduces exactly to \eqref{radst}.

The proof of Theorem~\ref{thm4} proceeds as follows. We denote
$$
T:=\max_{\Omega}u=\left\|u\right\|_{L^\infty (\Omega)}
\quad\text{ and }\quad \Gamma_s:=\left\{x\in\Omega:u(x)=s\right\}
$$
for $s\in(0,T)$. By Sard's theorem, almost every $s\in(0,T)$ is a regular value of 
$u$. For these regular values, we thus can apply Proposition~\ref{Sobolev} with $M=\Gamma_s$. 

We first use Proposition~\ref{Prop-semi-st} with  
$\eta=\varphi(u)$, 
where $\varphi$ is a Lipschitz function in $\left[0,T\right]$ 
with $\varphi(0)=0$.
The right-hand side of \eqref{semi1} becomes
\begin{eqnarray*}
\int_{\Omega}\left|\nabla u\right|^2\left|\nabla\eta\right|^2dx&=&
\int_{\Omega}\left|\nabla u\right|^4\varphi'(u)^2dx\\
&=&\int_{0}^{T}\left(\int_{\Gamma_s}\left|\nabla u\right|^3\,dV_s
\right)\varphi'(s)^2\,ds,
\end{eqnarray*}
by the coarea formula. Thus, \eqref{semi1} can be written as
\begin{eqnarray*}
&&\hspace{-1cm}\int_{0}^{T}\left(\int_{\Gamma_s}\left|\nabla u\right|^3\,dV_s
\right)\varphi'(s)^2\,ds\\
&&\geq\int_{\left\{\left|\nabla 
u\right|>0\right\}}\left(\left|\nabla_T\left|
\nabla u\right|\right|^2+\left|A\right|^2\left|\nabla 
u\right|^2\right)\varphi(u)^2dx\\
&&=\int_{0}^{T}\left(\int_{\Gamma_s\cap\left\{\left|\nabla 
u\right|>0\right\}}\frac{1}{\left|\nabla u\right|}
\left(\left|\nabla_T\left|
\nabla u\right|\right|^2+\left|A\right|^2\left|\nabla 
u\right|^2\right)\,dV_s\right)\varphi(s)^2\,ds\\
&&=\int_{0}^{T}\left(\int_{\Gamma_s\cap\left\{\left|\nabla 
u\right|>0\right\}}
\left( 4\left|\nabla_T\left|
\nabla u\right|^{1/2}\right|^2+\left(\left|A\right|\left|\nabla 
u\right|^{1/2}\right)^{2}\right) \,dV_s\right)\varphi(s)^2\,ds .
\end{eqnarray*}
We conclude that
\begin{equation}
\int_0^T h_1(s) \varphi(s)^2 \,ds
\leq
\int_0^T h_2(s)  \varphi'(s)^2 \,ds,
\label{semi3}
\end{equation}
for all Lipschitz functions 
$\varphi:\left[0,T\right]\rightarrow\mathbb{R}$ with 
$\varphi(0)=0$, where 
\begin{equation}\label{eq:27}
h_1(s):=\int_{\Gamma_s} 
\left( 4|\nabla_T |\nabla u|^{1/2}|^2 +\left( |A||\nabla u|^{1/2} \right)^2\right) \,dV_s\, ,\quad
h_2(s):=\int_{\Gamma_s} |\nabla u|^3\,dV_s 
\end{equation}
for every regular value $s$ of $u$. 

Inequality \eqref{semi3}, with $h_1$ and $h_2$ as defined above, will lead 
to our $L^{\infty}$ estimate of Theorem \ref{thm4} after choosing an 
appropriate test function $\varphi$ in \eqref{semi3}.

The rest of the proof differs in every dimension $n=2,3,$ and 4. But in  
all three cases it will be useful to denote 
\begin{equation}\label{defst}
B_t:=\frac{1}{t^2}\int_{\left\{u<t\right\}}\left|\nabla u\right|^4dx =
\frac{1}{t^2}\int_0^th_2(s)\,ds,
\end{equation}
where $t>0$ is a given positive constant as in the statement of the 
theorem. Note that the quantity $B_t$ is the main part of the right-hand 
side of our estimate \eqref{est4}. Let us start with the

\smallskip

\underline{{\it Cases} $n=2$ {\it and} 3}. In these dimensions we take the simple test function 
\begin{equation*}\label{eq:212}
\varphi(s)=\left\{
\begin{array}{lll}
s/t&\textrm{if}&s\leq t\\
1 &\textrm{if}&s>t.
\end{array}
\right.
\end{equation*}
With this choice of $\varphi$ and since 
$h_1(s)\geq\int_{\Gamma_s}\left|A\right|^2\left|\nabla u\right|\,dV_s\ $ 
---see definition \eqref{eq:27}---, inequality \eqref{semi3} leads to
\begin{equation}\label{eq:224}
\int_t^T\int_{\Gamma_s}
\left|A\right|^2\left|\nabla u\right|\,dV_s\,ds
\leq \int_0^th_2(s)\frac{1}{t^2}\,ds = \frac{1}{t^2}
\int_{\left\{u<t\right\}}\left|\nabla u\right|^4\,dx=:B_t.
\end{equation}

Next, we use a well known geometric inequality for the curve 
$\Gamma_s$ ($n=2$) or the surface $\Gamma_s$ ($n=3$). It also 
holds in every dimension $n\geq2$ and it states
\begin{equation}\label{eq:225}
\left|\Gamma_s\right|^{\frac{n-2}{n-1}}\leq C_n \int_{\Gamma_s}
\left|H\right|\,dV_s,
\end{equation}
where $H$ is the mean curvature of $\Gamma_s$, $C_n$ is a constant 
depending only on $n$, and $s$ is a regular value of $u$. In 
dimension $n=2$ this simply follows from the Gauss-Bonnet formula. 
For $n\geq3$, \eqref{eq:225} follows from the Michael-Simon and Allard Sobolev 
inequality (Proposition \ref{Sobolev} above) applied with $p=1$ and $v\equiv 1$.

We also use the classical isoperimetric inequality,
\begin{equation}\label{eq:226}
V(s):=\left|\left\{u>s\right\}\right|\leq 
C_n\left|\Gamma_s\right|^{\frac{n}{n-1}}.
\end{equation}
Now, \eqref{eq:225} and 
\eqref{eq:226} lead to 
\begin{equation*}
V(s)^{\frac{n-2}{n}}\leq C_n \int_{\Gamma_s}\left|H\right|\,dV_s
\leq C_n\left(\int_{\Gamma_s}\left|A\right|^2\left|\nabla u\right| 
\,dV_s\right)^{1/2}\left(\int_{\Gamma_s}\frac{dV_s}
{\left|\nabla u\right|}\right)^{1/2}
\end{equation*}
for all regular values $s$,
by Cauchy-Schwarz and since $\left|H\right|\leq\left|A\right|$. 
{From} this, we deduce 
\begin{eqnarray}\label{eq:227}
T-t=\int_t^T \,ds&\leq& \\ 
&  & \hspace{-3cm} \leq \ \int_t^T C_n\left(\int_{\Gamma_s}
\left|A\right|^2\left|\nabla u\right| 
\,dV_s\right)^{1/2}\left(V(s)^{\frac{2(2-n)}{n}}
\int_{\Gamma_s}\frac{dV_s}
{\left|\nabla u\right|}\right)^{1/2} \,ds\nonumber \\
\nonumber
&  & \hspace{-3cm} \leq \ C_n\left(\int_t^T \int_{\Gamma_s}
\left|A\right|^2\left|\nabla u\right| 
\,dV_s\,ds\right)^{1/2}\left(\int_t^T V(s)^{\frac{2(2-n)}{n}}
\int_{\Gamma_s}\frac{dV_s}
{\left|\nabla u\right|}\,ds\right)^{1/2}\\
\label{eq:227bis}
&  & \hspace{-3cm} \leq \ C_nB_t^{1/2}\left(\int_t^T
V(s)^{\frac{2(2-n)}{n}}\int_{\Gamma_s}\frac{dV_s}
{\left|\nabla u\right|}\,ds\right)^{1/2},
\end{eqnarray}
where we have used \eqref{eq:224} in the last inequality. 

Finally, since $V(s)=\left|\left\{u>s\right\}\right|$ is a 
nonincreasing function, it is differentiable almost everywhere 
and, by the coarea formula, 
\begin{equation*}
-V'(s)=\int_{\Gamma_s}\frac{dV_s}{\left|\nabla u\right|}
\qquad \text{for a.e.}\ s\in(0,T).
\end{equation*}
In addition, for $n\leq3$, $V(s)^{\frac{4-n}{n}}$ is 
nonincreasing in $s$ and thus its total variation satisfies
\begin{eqnarray*}
\left|\Omega\right|^{\frac{4-n}{n}}&\geq& V(t)^{\frac{4-n}{n}}
=\left[V(s)^{\frac{4-n}{n}}\right]^{s=t}_{s=T}\\
&\geq&\int_t^T \frac{4-n}{n} V(s)^{\frac{2(2-n)}{n}}
\left(-V'(s)\right)\,ds \\
&=&\frac{4-n}{n}\int_t^T V(s)^{\frac{2(2-n)}{n}}\int_{\Gamma_s}
\frac{dV_s}{\left|\nabla u\right|}\,ds.
\end{eqnarray*}
{From} this, \eqref{eq:227}, and \eqref{eq:227bis}, we conclude the desired inequality 
\begin{equation}
T-t\leq C_n B_t^{1/2} \left|\Omega\right|^{(4-n)/(2n)}
\end{equation}
for $n\leq3$.

Note that this argument gives nothing for $n\geq 4$ since the
integral in \eqref{eq:227bis}, 
\begin{equation}
\int_t^T V(s)^{\frac{2(2-n)}{n}}\left(-V'(s)\right)\,ds 
= \int_0^{V(t)} \frac{dr}{r^{\frac{2(n-2)}{n}}},
\end{equation}
is not convergent at $s=T$ (i.e., $r=0$) because $2(n-2)/n\geq1$.

We now turn to the 

\smallskip

\underline{\it Case $n=4$}. For $n\ge 4$, we apply the Michael-Simon and Allard 
Sobolev inequality \eqref{MSsob} with $M=\Gamma_s$, $p=2<m=n-1$, and 
$v= \left|\nabla u\right|^{1/2}$. Note that we have  $\left|H\right|\leq\left|A\right|$. 
Recalling \eqref{eq:27}, we obtain
\begin{equation}\label{estimate2}
\left( \int_{\Gamma_s} |\nabla u|^\frac{n-1}{n-3} \ dV_s \right)^\frac{n-3}
{n-1}\leq C_n h_1(s)
\end{equation}
for all regular values $s$ of $u$. This estimate combined with \eqref{semi3} leads to 
\begin{equation}\label{estgenn}
\int_{0}^{T}\left(\int_{\Gamma_s} |\nabla u|^\frac{n-1}{n-3} \,dV_s 
\right)^\frac{n-3}{n-1}\varphi(s)^2\,ds\leq C_n\int_{0}^{T}\left(\int_{\Gamma_s} 
|\nabla u|^3 \,dV_s \right)\varphi'(s)^2\,ds
\end{equation}
for all Lipschitz functions $\varphi$ in $[0,T]$ with $\varphi(0)=0$. We only know how 
to derive an $L^{\infty}$ estimate for $u$ (i.e., a bound on 
$T=\max_{\Omega}u$) from \eqref{estgenn} when the exponent $(n-1)/(n-3)$ on 
its left-hand side is larger than or equal to the one on the right-hand side, 
i.e., 3. That is, we need $(n-1)/(n-3)\geq3$, which means $n\leq4$.

Therefore, taking $n=4$, then $(n-1)/(n-3)=3$ and \eqref{estimate2} becomes
\begin{equation}\label{cont4}
h_2^{1/3}\leq C h_1 \qquad\text{a.e. in}\ (0,T)\ (\text{when}\ n=4),
\end{equation}
where $C$ is a universal constant. 

To proceed, the idea is to choose a function $\varphi$ that tries to violate the inequality \eqref{semi3},
by imposing $h_1\varphi^2= 2 h_2 \varphi'^2$. This ODE will lead to the choice \eqref{eq:211} below. 
The ODE will not be satisfied in all $[0,T]$ since we must have $\varphi (0)=0$. For this reason, 
\eqref{semi3} will not be violated, but instead will lead to an estimate.

For every regular value $s$ of $u$, we 
have $0<h_2(s)$ and $h_1(s)<\infty$ (simply by their definition). 
This together with \eqref{cont4} gives $h_1/h_2\in\left(0,+\infty\right)$ 
a.e. in $\left(0,T\right)$. Thus, defining 
$$g_k(s):=\min\left( k,\frac{h_1(s)}{h_2(s)}\right)$$
for regular values $s$ and for a positive integer $k$, we have that 
$g_k\in L^{\infty}(0,T)$ and 
\begin{equation}\label{gknicr}
g_k(s)\nearrow\frac{h_1(s)}{h_2(s)}\in(0,+\infty) \qquad\text{as}\ 
k\uparrow\infty,\ \text{for}\ \text{a.e.}\ s\in(0,T).
\end{equation}
Since $g_k\in L^{\infty}(0,T)$, the function 
\begin{equation}\label{eq:211}
\varphi_k(s):=\left\{
\begin{array}{lll}
s/t&\textrm{if}&s\leq t,\\
\displaystyle \exp\left(
\frac{1}{\sqrt{2}}\int_t^s \sqrt{g_k(\tau)}\ d\tau\right)
&\textrm{if}&t<s\leq T,
\end{array}
\right.
\end{equation}
is well defined, Lipschitz in $\left[0,T\right]$, and satisfies 
$\varphi_k(0)=0$.

Since
$$h_2(\varphi_k')^2=h_2\frac{1}{2}g_k\varphi_k^2\leq\frac{1}{2}
h_1\varphi_k^2\qquad \text{in}\ (t,T),$$
\eqref{semi3} used with $\varphi=\varphi_k$ leads to
\begin{equation}\label{eq:213}
\int_t^T h_1\varphi_k^2 \,ds\leq \frac{2}{t^2}\int_0^t h_2\,ds=
\frac{2}{t^2}\int_{\{u< t\}}|\nabla u|^4\  dx=2 B_t.
\end{equation}
Recall that $B_t$ was defined in \eqref{defst} and that we 
need to establish $T-t\leq C B_t^{1/2}$. 

By \eqref{gknicr} we have
\begin{equation}\label{eq:214}
T-t=\int_t^T\,ds=\sup_{k\geq1}\int_t^T\sqrt[4]{\frac{h_2g_k}{h_1}}\,ds.
\end{equation}
Using \eqref{eq:213} and Cauchy-Schwarz, we have that
\begin{eqnarray}\label{eq:218}
\int_t^T\sqrt[4]{\frac{h_2g_k}{h_1}}\,ds&=&\int_t^T\left(\sqrt{h_1}
\varphi_k\right)\left(\sqrt[4]{\frac{h_2g_k}{h_1^3}}\frac{1}
{\varphi_k}\right)\,ds\\
&\leq& \left(2B_t\right)^{1/2}\left(\int_t^T\sqrt{\frac{h_2g_k}{h_1^3}}
\frac{1}{\varphi_k^2}\,ds\right)^{1/2} \nonumber \\ 
\label{2.20}
&\leq& \left(2B_t\right)^{1/2}\left(C\int_t^T\sqrt{g_k}
\frac{1}{\varphi_k^2}\,ds\right)^{1/2}. 
\end{eqnarray}
In the last inequality we have used our crucial estimate \eqref{cont4}.

Finally we bound the integral in \eqref{2.20}, using the definition 
\eqref{eq:211} of $\varphi_k$, as follows:
\begin{eqnarray*}
\int_t^T\sqrt{g_k}\frac{1}{\varphi_k^2}\,ds &=& \int_t^T\sqrt{g_k}
\frac{1}{\varphi_k^2}\frac{\varphi_k'}{\frac{1}{\sqrt{2}}
\sqrt{g_k}\varphi_k}\,ds\\
&=& \sqrt{2}\int_t^T\frac{\varphi'_k}{\varphi_k^3}\,ds = 
\frac{\sqrt{2}}{2}\left[\varphi_k^{-2}(s)\right]_{s=T}^{s=t} \\
&\leq&\frac{\sqrt{2}}{2}\varphi_k^{-2}(t)=\frac{\sqrt{2}}{2}.  
\end{eqnarray*}
This bound, together with \eqref{eq:214},\eqref{eq:218}, and \eqref{2.20}, finishes
the proof in dimension~$4$.

\end{document}